\documentclass{amsart}

\usepackage{cmap}

\usepackage[english]{babel}
\usepackage[OT2,T1]{fontenc}

\usepackage{colonequals}
\usepackage{booktabs}
\usepackage{mathrsfs} 
\usepackage{amssymb}

\newtheorem{theorem}{Theorem}

\usepackage{url}

\usepackage[usenames,dvipsnames]{xcolor} 
\definecolor{darkgreen}{rgb}{0,0.5,0}
\definecolor{darkblue}{rgb}{0,0,0.8}
\definecolor{darkred}{rgb}{0.8,0,0}

\usepackage[pdfencoding=auto,colorlinks,citecolor=darkgreen,linkcolor=darkblue,urlcolor=darkred]{hyperref}

\usepackage[step=1,stretch=30,shrink=30,selected=true]{microtype}

\DeclareSymbolFont{cyrletters}{OT2}{wncyr}{m}{n}
\DeclareMathSymbol{\Sha}{\mathalpha}{cyrletters}{"58}

\renewcommand{\epsilon}{\varepsilon}
\renewcommand{\phi}{\varphi}
\renewcommand{\theta}{\vartheta}

\newcommand{\Q}{{\mathbb{Q}}}
\newcommand{\Qbar}{{\overline{\Q}}}
\newcommand{\Z}{{\mathbb{Z}}}

\newcommand{\F}{{\mathbb{F}}}

\newcommand{\R}{{\mathbb{R}}}

\newcommand{\pfr}{{\mathfrak{p}}}

\newcommand{\End}{\operatorname{End}}

\newcommand{\Aut}{\operatorname{Aut}}
\newcommand{\Jac}{\operatorname{Jac}}
\newcommand{\GL}{\operatorname{GL}}
\newcommand{\SL}{\operatorname{SL}}

\newcommand{\Gal}{\operatorname{Gal}}
\newcommand{\GalQ}{{\Gal(\Qbar|\Q)}}

\renewcommand{\H}{{\operatorname{H}}}

\newcommand{\Ocal}{{\mathcal{O}}}

\newcommand{\Sel}{\operatorname{Sel}}

\newcommand{\disc}{\operatorname{disc}}

\newcommand{\rk}{\operatorname{rk}}

\newcommand{\lcm}{\operatorname{lcm}}

\newcommand{\an}{{\mathrm{an}}}
\newcommand{\nr}{{\mathrm{nr}}}

\newcommand{\tors}{{\mathrm{tors}}}

\newcommand{\Reg}{\operatorname{Reg}}
\newcommand{\ord}{\operatorname{ord}}

\newcommand{\isom}{\cong}

\begin{document}

\title[Exact verification of strong BSD for some modular abelian surfaces]{Exact verification of the strong BSD conjecture for some absolutely simple abelian surfaces}

\author[T. Keller]{Timo Keller}\email{Timo.Keller@uni-bayreuth.de} 
\author[M. Stoll]{Michael Stoll}\email{Michael.Stoll@uni-bayreuth.de} 
\address{Lehrstuhl Mathematik~II (Computeralgebra), Universität Bayreuth, Universitätsstraße 30, 95440 Bayreuth, Germany}

\thanks{This work was supported by the Deutsche Forschungsgemeinschaft (DFG, German Research Foundation) -- Projektnummer STO 299/18-1, AOBJ: 667349.}

\date{June 30, 2021}

\subjclass[2020]{11G40 (11-04, 11G10, 11G30, 14G35)}
	
	\begin{abstract}
		Let $X$ be one of the $28$ Atkin-Lehner quotients of a curve~$X_0(N)$
		such that $X$ has genus~$2$ and its Jacobian variety~$J$ is
		absolutely simple. We show that the Shafarevich-Tate group $\Sha(J/\Q)$
		is trivial. This verifies the strong BSD~conjecture for~$J$.
	\end{abstract}

	\maketitle

	\section{Introduction}
	
	Let $A$ be an abelian variety over~$\Q$ and assume that its $L$-series
	$L(A,s)$ admits an analytic continuation to the whole complex plane.
	The \emph{weak BSD conjecture} (or BSD rank conjecture) predicts that
	the Mordell-Weil rank $r = \rk A(\Q)$ of~$A$ equals the analytic rank
	$r_{\an} = \ord_{s=1} L(A, s)$. The \emph{strong BSD conjecture}
	asserts that the Shafarevich-Tate group $\Sha(A/\Q)$ is finite and that
	its order equals the ``analytic order of Sha'',
	\begin{equation} \label{Sha_an}
		\#\Sha(A/\Q)_{\an} := \frac{\#A(\Q)_{\tors} \cdot \#A^\vee(\Q)_{\tors}}{\prod_v c_v}
		\cdot \frac{L^*(A, 1)}{\Omega_A \Reg_{A/\Q}} \,.
	\end{equation}
	Here $A^\vee$ is the dual abelian variety, $A(\Q)_{\tors}$ denotes the
	torsion subgroup of~$A(\Q)$, the product $\prod_v c_v$ runs over all finite
	places of~$\Q$ and $c_v$ is the Tamagawa number of~$A$ at~$v$, $L^*(A, 1)$
	is the leading coefficient of the Taylor expansion of~$L(A, s)$ at $s =
	1$,
	and $\Omega_A$ and $\Reg_{A/\Q}$ denote the volume of~$A(\R)$ and the
	regulator of~$A(\Q)$, respectively.
	
	If $A$ is \emph{modular} in the sense that $A$ is an isogeny factor
	of the Jacobian~$J_0(N)$ of the modular curve~$X_0(N)$ for some~$N$,
	then the analytic continuation of~$L(A,s)$ is known. If $A$ is in addition
	absolutely simple, then $A$ is associated (up to isogeny) to a Galois orbit
	of size~$\dim(A)$ of newforms of weight~$2$ and level~$N$, such that
	$L(A, s)$ is the product of $L(f, s)$ with $f$ running through these newforms.
	Such an abelian variety has \emph{real multiplication}: its endomorphism
	ring (over~$\Q$ and over~$\Qbar$) is an order in a totally real
	number field of degree~$\dim(A)$. If, furthermore,
	$\ord_{s=1} L(f,s) \in \{0,1\}$ for one (equivalently, all) such $f$,
	then the weak BSD conjecture holds for~$A$; see~\cite{KolyvaginLogachev}.
	
	All elliptic curves over~$\Q$ arise as one-dimensional modular abelian
	varieties~\cite{Wiles1995,TaylorWiles1995,BCDT2001} such that $N$ is the conductor
	of~$A$. For all elliptic curves of (analytic) rank $\le 1$ and $N < 5000$,
	the strong BSD conjecture has been verified~\cite{GJPST,MillerStoll2013,CreutzMiller2012}.
	
	In this note, we consider certain absolutely simple abelian \emph{surfaces}
	and show that strong BSD holds for them. One class of such surfaces arises
	as the Jacobians of quotients~$X$ of~$X_0(N)$ by a group of Atkin-Lehner
	operators. Hasegawa~\cite{Hasegawa1995} has determined the complete
	list of such~$X$ of genus~$2$; $28$ of them have absolutely simple
	Jacobian~$J$. For most of these Jacobians (and those of further curves taken
	from~\cite{Wang1995}), it has been numerically verified
	in~\cite{FLSSSW,vanBommel2019} that $\#\Sha(J/\Q)_{\an}$ is very close to
	an integer,
	which equals $\#\Sha(J/\Q)[2]$ ($= 1$ in the cases
	considered here). We complete the verification of strong BSD for
	these Jacobians by showing that $\#\Sha(J/\Q)_{\an}$ is indeed an integer
	and $\Sha(J/\Q)$ is trivial.
	
	\section{Methods and algorithms} \label{S:algo}
	
	In the following, we denote the abelian surface under consideration by~$A$;
	it is an absolutely simple isogeny quotient of~$J_0(N)$, defined over~$\Q$.
	We frequently use the fact that $A$ can be obtained as the Jacobian variety
	of a curve~$X$ of genus~$2$. The algorithms described below have been
	implemented in Magma~\cite{Magma}.
	
	Recall that a \emph{Heegner discriminant} for~$A$ is a fundamental discriminant
	$D < 0$ such that for $K = \Q(\sqrt{D})$, the analytic rank of~$A/K$ equals
	$\dim A = 2$ and all prime divisors of~$N$ split in~$K$. Heegner discriminants exist
	by~\cite{BFH1989,Waldspurger1985}. Since Magma can determine whether
	$\ord_{s=1} L(f, s)$ is $0$, $1$, or larger (for a newform~$f$ as considered
	here), we can easily find one or several Heegner discriminants for~$A$.
	
	Associated to each Heegner discriminant~$D$ is a \emph{Heegner point}
	$y_D \in A(K)$, unique up to sign and adding a torsion point.
	In particular, the \emph{Heegner index} $I_D = (A(K) : \End(A) \cdot y_D)$
	is well-defined.
	
	Recall that $\Ocal = \End(A) = \End_{\Q}(A)$ is an order in a real quadratic field.
	In all cases considered here, $\Ocal$ is a maximal order and a principal ideal domain.
	For each prime ideal~$\pfr$ of~$\Ocal$, we have the residual Galois
	representation $\rho_\pfr \colon \GalQ \to \Aut(A[\pfr]) \simeq \GL_2(\F_\pfr)$,
	where $\F_\pfr = \Ocal/\pfr$ denotes the residue class field.
	
	We can use Magma's functionality for $2$-descent on hyperelliptic
	Jacobians based on~\cite{Stoll2001} to determine $\Sha(A/\Q)[2]$.
	In all cases considered here,
	this group is trivial, which implies that $\Sha(A/\Q)[2^\infty] = 0$.
	(In fact, this had already been done in~\cite{FLSSSW} for most of the curves.)
	It is therefore sufficient to consider the $p$-primary parts of~$\Sha(A/\Q)$
	for odd~$p$.
	
	\begin{theorem}\label{th:main}
		Let $A$ be an abelian variety of $\GL_2$-type over $\Q$. Assume that $\ord_{s=1} L(f,s) \in \{0,1\}$ for one (equivalently, all) newform associated to $A$.
		\begin{enumerate}
			\item \label{item:CCSS} If the level $N$ of $A$ is square-free, $\ord_p(\#\Sha(A/\Q)_\an) = \ord_p(\#\Sha(A/\Q))$ for all rational primes $p \neq 2$ such that $\rho_\pfr$ is irreducible for all $\pfr \mid p$.
			
			\item \label{item:KL} If there exists a polarization $\lambda: A \to A^\vee$, $\Sha(A/\Q)[\pfr] = 0$ for all prime ideals $\pfr \mid p \neq 2$ such that $\rho_\pfr$ is irreducible and $p$ does not divide $\deg\lambda$, and, for some Heegner field $K$ with Heegner discriminant $D$, $I_D$ and the order of the groups $\H^1(K_v^\nr|K_v,A)$ with $v$ running through the places of $K$.
		\end{enumerate}
	\end{theorem}
	\begin{proof}
		(\ref{item:CCSS}) is~\cite[Theorems~C and~D]{CCSS}. (\ref{item:KL}) is an explicit version of~\cite{KolyvaginLogachev}.
	\end{proof}
	
	We have implemented the following algorithms.
	\begin{enumerate}
		\item \emph{Image of the residual Galois representations.}
		Extending the algorithm described in~\cite{Dieulefait2002}, which determines
		a finite small superset of the primes $p$ with $\rho_p$ reducible in the
		case that $\End_\Qbar(A) = \Z$, we obtain a finite small superset of the
		prime ideals~$\pfr$ of~$\Ocal$ such that $\rho_\pfr$ is reducible. Building
		upon this and~\cite{Cojocaru2005}, we can also check whether $\rho_\pfr$
		has maximal possible image $\GL_2(\F_\pfr)^{\det \in \F_p^\times}$.
		
		The irreducibility of~$\rho_\pfr$ for all $\pfr \mid p$ is the crucial
		hypothesis in~\cite[Theorems~C and~D]{CCSS}, and
		in~\cite{KolyvaginLogachev}.
		
		\item \emph{Computation of the Heegner index.} 
		We can compute the height of a Heegner point using the main theorem
		of~\cite{GrossZagier1986}. By enumerating all points of that approximate height using~\cite{MuellerStoll2016b}, we can identify the Heegner point $y_D \in A(K)$
		as a $\Q$-point on $A$, or on the quadratic twist $A^K$, depending on the
		analytic rank of~$A/\Q$. An alternative implementation uses the $j$-invariant
		morphism $X_0(N) \to X_0(1)$ and takes the preimages of the $j$-invariants
		belonging to elliptic curves with CM by the order of discriminant~$D$.
		A variant of this is based on approximating $q$-expansions of cusp forms
		analytically and finding the Heegner point as an algebraic approximation.
		
		\item \label{End}
		\emph{Determination of the (geometric) endomorphism ring of~$A/\Q$ and its action
			on the Mordell-Weil group $A(\Q)$.} Given the Heegner point~$y_D$, this
		can be used to compute the Heegner index~$I_D$.
		
		We can also compute the \emph{kernel of a given endomorphism} as an abstract
		$\GalQ$-module together with explicit generators in~$A(\Qbar)$.
		We apply this to find the characters corresponding to the constituents
		of~$\rho_\pfr$ when the representation is reducible.
		
		\item \emph{Analytic order of $\Sha$.}
		If the $L$-rank $\ord\nolimits_{s=1}L(f,s)$ of $A/\Q$ is zero, then we can compute
		$\#\Sha(A/\Q)_\an$ exactly as a rational number using modular symbols via
		Magma's \verb|LRatio| function, which gives $L(A,1)/\Omega_A^{-1}
		\in \Q_{>0}$,
		together with~\eqref{Sha_an}, since $\#A(\Q)_\tors = \#A^\vee(\Q)_\tors$
		and the Tamagawa numbers~$c_v$ are known.
		
		When the $L$-rank is~$1$, we can compute the analytic order of~$\Sha$ from $\#\Sha(A/K)_\an = \#\Sha(A/\Q)_\an \cdot \#\Sha(A^K/\Q)_\an \cdot 2^\text{(bounded exponent)}$ and the
		formula
		\begin{align*}
			\#\Sha(A/K)_\an 
			&=	\frac{\#A(K)_\tors \#A^\vee(K)_\tors}{c_\pi^2 u_K^4 \prod_p c_p(A/\Q)^2}
			\cdot \frac{\|\omega_f\|^2 \|\omega_{f^\sigma}\|^2}{\Omega_{A/K}}
			\cdot \frac{\hat{h}(y_{D,f}) \hat{h}(y_{D,f^\sigma})\disc\Ocal}{\Reg_{A/K}}
		\end{align*}
		deduced from~\cite{GrossZagier1986}; here, the last two factors are integral.
		In the computation of $\#\Sha(A^K/\Q)_\an$, we use van Bommel's code to compute the
		Tamagawa numbers of $A/\Q$ and $A^K/\Q$ and the real period of $A^K/\Q$. In the cases where
		his code did not succeed, we used another Heegner discriminant.
		
		\item \label{Sel}
		\emph{Isogeny descent.}
		In the cases when $\pfr$ is odd and $\rho_\pfr$ is reducible, we determined
		characters $\chi_1$ and~$\chi_2$ such that
		\[ \rho_\pfr \isom \begin{pmatrix} \chi_1 & * \\ 0 & \chi_2 \end{pmatrix} ; \]
		see~\eqref{End} above. We then compute upper bounds for the $\F_p$-dimensions
		of the two Selmer groups associated to the corresponding two isogenies of
		degree~$p$ whose composition is multiplication by a generator~$\pi$ of~$\pfr$ on~$A$;
		see~\cite{SchaeferStoll2004}.
		From this, we deduce an upper bound for the dimension of the $\pi$-Selmer
		group of~$A$, which, in the cases considered here, is always $\le
		1$.
		Using the known finiteness of~$\Sha(A/\Q)$, which implies that $\Sha(A/\Q)[\pfr]$
		has even dimension, this shows that $\Sha(A/\Q)[\pfr] = 0$.
		
		\item \label{p-adic}
		\emph{Computation of the $p$-adic $L$-function.}
		We can also compute the $p$-adic $L$-functions of newforms of weight $2$,
		trivial character and arbitrary coefficient ring for $p^2 \nmid N$.
		Computing $\ord_\pfr\mathcal{L}_\pfr^*(f,0)$ and using the known
		results~\cite{SkinnerUrban2014,Skinner2016} about the $\GL_2$~Iwasawa Main Conjecture~(IMC) with the hypotheses that
		$\rho_\pfr$ is irreducible and there is a $q \| N$ with $\rho_\pfr$ ramified
		at $q \neq p$ gives us information about the $\pfr^\infty$-Selmer group.
	\end{enumerate}
	
	\section{Results}
	
	\begin{figure}[tbp]
		\begin{tabular}{lllllrrr}
			\toprule
			$X$ & $r$ & $\Ocal$ & $\#\Sha_\an$ & $\rho_\pfr$ red. & $c$ & $(D,I_D)$ &  $\#\Sha$ \\
			\midrule
			$X_0(23)$        & $0$ & $\sqrt{5}$  & $1$   & $11_1$           & $11$ & $(-7, 11)$  & $11^{0}$ \\
			$X_0(29)$        & $0$ & $\sqrt{2}$  & $1$   & $7_1$             & $7$ & $(-7,  7)$  & $7^{0}$  \\
			$X_0(31)$        & $0$ & $\sqrt{5}$  & $1$   & $\sqrt{5}$        & $5$ & $(-11, 5)$  & $5^0$  \\
			$X_0(35)/w_7$    & $0$ & $\sqrt{17}$ & $1$   & $2_1$             & $1$ & $(-19, 1)$  & $1$    \\
			$X_0(39)/w_{13}$ & $0$ & $\sqrt{2}$  & $1$   & $\sqrt{2}$, $7_1$ & $7$ & $(-23, 7)$  & $7^{0}$  \\
			$X_0(67)^+$      & $2$ & $\sqrt{5}$  & $1$   &                   & $1$ & $(-7,  1)$  & $1$    \\
			$X_0(73)^+$      & $2$ & $\sqrt{5}$  & $1$   &                   & $1$ & $(-19, 1)$  & $1$    \\
			$X_0(85)^*$      & $2$ & $\sqrt{2}$  & $1$   & $\sqrt{2}$        & $1$ & $(-19, 1)$  & $1$    \\
			$X_0(87)/w_{29}$ & $0$ & $\sqrt{5}$  & $1$   & $\sqrt{5}$        & $5$ & $(-23, 5)$  & $5^0$  \\
			$X_0(93)^*$      & $2$ & $\sqrt{5}$  & $1$   &                   & $1$ & $(-11, 1)$  & $1$    \\
			$X_0(103)^+$     & $2$ & $\sqrt{5}$  & $1$   &                   & $1$ & $(-11, 1)$  & $1$    \\
			$X_0(107)^+$     & $2$ & $\sqrt{5}$  & $1$   &                   & $1$ & $(-7,  1)$  & $1$    \\
			$X_0(115)^*$     & $2$ & $\sqrt{5}$  & $1$   &                   & $1$ & $(-11, 1)$  & $1$    \\
			$X_0(125)^+$     & $2$ & $\sqrt{5}$  & $1$   & $\sqrt{5}$        & $1$ & $(-11, 1)$  & $5^0$  \\
			$X_0(133)^*$     & $2$ & $\sqrt{5}$  & $1$   &                   & $1$ & $(-31, 1)$  & $1$    \\
			$X_0(147)^*$     & $2$ & $\sqrt{2}$  & $1$   & $\sqrt{2}$, $7_1$ & $1$ & $(-47, 1)$  & $7^0$  \\
			$X_0(161)^*$     & $2$ & $\sqrt{5}$  & $1$   &                   & $1$ & $(-19, 1)$  & $1$    \\
			$X_0(165)^*$     & $2$ & $\sqrt{2}$  & $1$   & $\sqrt{2}$        & $1$ & $(-131,1)$  & $1$    \\
			$X_0(167)^+$     & $2$ & $\sqrt{5}$  & $1$   &                   & $1$ & $(-15, 1)$  & $1$    \\
			$X_0(177)^*$     & $2$ & $\sqrt{5}$  & $1$   &                   & $1$ & $(-11, 1)$  & $1$    \\
			$X_0(191)^+$     & $2$ & $\sqrt{5}$  & $1$   &                   & $1$ & $(-7,  1)$  & $1$    \\
			$X_0(205)^*$     & $2$ & $\sqrt{5}$  & $1$   &                   & $1$ & $(-31, 1)$  & $1$    \\
			$X_0(209)^*$     & $2$ & $\sqrt{2}$  & $1$   &                   & $1$ & $(-51, 1)$  & $1$    \\
			$X_0(213)^*$     & $2$ & $\sqrt{5}$  & $1$   &                   & $1$ & $(-11, 1)$  & $1$    \\
			$X_0(221)^*$     & $2$ & $\sqrt{5}$  & $1$   &                   & $1$ & $(-35, 1)$  & $1$    \\
			$X_0(287)^*$     & $2$ & $\sqrt{5}$  & $1$   &                   & $1$ & $(-31, 1)$  & $1$    \\
			$X_0(299)^*$     & $2$ & $\sqrt{5}$  & $1$   &                   & $1$ & $(-43, 1)$  & $1$    \\
			$X_0(357)^*$     & $2$ & $\sqrt{2}$  & $1$   &                   & $1$ & $(-47, 1)$  & $1$    \\
			\bottomrule
		\end{tabular}
		\caption{BSD data for the absolutely simple modular Jacobians of Atkin-Lehner quotients of $X_0(N)$.}
		\label{fig:BSDdata}
	\end{figure}
	
	\begin{figure}[tbp]
		\begin{tabular}{lrrrr}
			\toprule
			$X$ & $D_K$ & $\#\Sha(A^K/\Q)_\an$ & $\#\Sha(A/K)_\an$ & $\#\Sha(A/\Q)_\an$\\
			\midrule
			$X_0(67)^+$		 & $ -7$   & $4$  & $1$ & $1$ \\
			$X_0(73)^+$      & $-19$   & $4$  & $1$ & $1$ \\
			$X_0(85)^*$      & $-19$   & $4$  & $1$ & $1$ \\
			$X_0(93)^*$      & $-11$   & $1$  & $1$ & $1$ \\
			$X_0(103)^+$     & $-11$   & $4$  & $1$ & $1$ \\
			$X_0(107)^+$     & $ -7$   & $4$  & $1$ & $1$ \\
			$X_0(115)^*$     & $-11$   & $1$  & $1$ & $1$ \\
			$X_0(125)^+$     & $-11$   & $4$  & $1$ & $1$ \\
			$X_0(133)^*$     & $-31$   & $4$  & $1$ & $1$ \\
			$X_0(147)^*$     & $-47$   & $4$  & $1$ & $1$ \\
			$X_0(161)^*$     & $-19$   & $1$  & $1$ & $1$ \\
			$X_0(165)^*$     &$-131$   & $16$ & $4$ & $1$ \\
			$X_0(167)^+$     & $-15$   & $4$  & $1$ & $1$ \\
			$X_0(177)^*$     & $-11$   & $4$  & $1$ & $1$ \\
			$X_0(191)^+$     & $ -7$   & $4$  & $1$ & $1$ \\
			$X_0(205)^*$     & $-31$   & $4$  & $1$ & $1$ \\
			$X_0(209)^*$     & $-79^*$ & $2$  & $1$ & $1$ \\
			$X_0(213)^*$     & $-11$   & $4$  & $1$ & $1$ \\
			$X_0(221)^*$     & $-35$   & $4$  & $1$ & $1$ \\
			$X_0(287)^*$     & $-21$   & $4$  & $1$ & $1$ \\
			$X_0(299)^*$     & $-43$   & $4$  & $1$ & $1$ \\
			$X_0(357)^*$     & $-47$   & $2$  & $1$ & $1$ \\
			\bottomrule
		\end{tabular}
		\caption{Analytic order of $\Sha$ for the curves of $L$-rank $1$. (A $^*$ means that we used a different Heegner discriminant than in Figure~\ref{fig:BSDdata} in the case where van Bommel's \texttt{TamagawaNumber} did not succeed.)}
		\label{fig:Sha_an}
	\end{figure}
	
	Our results are summarized in Figure~\ref{fig:BSDdata}. The first column
	gives the genus~$2$ curve~$X$ as a quotient of~$X_0(N)$ by a subgroup of the Atkin-Lehner
	involutions. We denote the Atkin-Lehner involution associated to
	a divisor~$d$ of~$N$ such that $d$ and~$N/d$ are coprime by~$w_d$.
	We write $X_0(N)^+$ for $X_0(N)/w_N$ and $X_0(N)^*$ for the
	quotient of~$X_0(N)$ by the full group of Atkin-Lehner operators.
	We are considering the Jacobian $A$ of $X$.
	
	The second column gives the algebraic rank of $A/\Q$, which is equal to its analytic rank
	by the combination of the main results of~\cite{GrossZagier1986} and~\cite{KolyvaginLogachev}.
	
	The third column specifies $\Ocal$ as the maximal order in the number field obtained
	by adjoining the given square root to~$\Q$.
	
	The fourth column gives the analytic order of the Shafarevich-Tate group of $A$,
	defined as in the introduction. For the surfaces of $L$-rank~$1$, the intermediate
	results of our computation are contained in Figure~\ref{fig:Sha_an}.
	
	The fifth column specifies the prime ideals~$\pfr$ of~$\Ocal$ such that $\rho_\pfr$
	is reducible. The notation $p_1$ means that $p$ is split in~$\Ocal$ and $\rho_\pfr$
	is reducible for exactly one $\pfr \mid p$. If $p$ is ramified in~$\Ocal$, we write
	$\sqrt{p}$ for the unique prime ideal~$\pfr \mid p$.
	
	The sixth column gives the odd part of~$\lcm_p c_p(A/\Q)$, which can be obtained from
	the LMFDB~\cite{LMFDB}.
	
	The seventh column gives a Heegner discriminant~$D$ for~$A$ together with 
	the odd part
	of the Heegner index $I_D$.
	Our computation confirms that the Tamagawa product divides the Heegner index.
	
	The last column contains the order of the Shafarevich-Tate group of $A/\Q$.
	An entry~$1$ means that it follows immediately from the previous columns,
	the computation of $\Sel_2(A/\Q)$ and Theorem~\ref{th:main}
	that all $p$-primary components of $\Sha(A/\Q)$ vanish.
	
	Otherwise, the order
	of~$\Sha(A/\Q)$ is given as a product of powers of the odd primes~$p$ such that
	some $\rho_\pfr$ with $\pfr \mid p$ is reducible or $p$ divides $c \cdot I_D$.
	(In each of these cases, there is exactly one such~$p$.)
	We have to justify that the exponents are all zero.
	In the first three rows we use~\cite{Mazur1978} to show that for the reducible
	odd~$\pfr$ on has $\Sha(J_0(p)/\Q)[\pfr] = 0$; this is a consequence of 
	these prime
	ideals being Eisenstein primes.
	
	In the remaining cases, we used the approach described in item~\eqref{Sel}
	in Section~\ref{S:algo}.
	For the rows with $A = \Jac(X_0(39)/w_{13})$ and $A = \Jac(X_0(87)/w_{29})$,
	one has non-split short exact sequences of Galois modules
	\begin{align*}
		0 \to \Z/p \to A[\pfr] \to \mu_p \to 1
	\end{align*}
	with $p = 7$ and~$5$, respectively. For the only two non-semistable abelian surfaces
	we found the following isomorphism and exact sequence.
	\begin{align*}
		J_0(125)^+[\sqrt{5}] &\isom \mu_5^{\otimes 2} \oplus \mu_5^{\otimes 3} \\
		1 \to \mu_7^{\otimes 4} \to \Jac(X_0(147)^*)[\pfr] &\to \mu_7^{\otimes 3} \to 1
	\end{align*}
	In all cases, we find that $\Sha(A/\Q)[\pfr] = 0$.
	Note that for the $\pfr \mid 7$ for which $\rho_\pfr$ is \emph{irreducible},
	$\Jac(X_0(147)^*)[\pfr] = 0$ follows from~\cite{KolyvaginLogachev}
	because $I_{-43}$ is not divisible by $7$. In the case of the square-free levels $N = 23, 29, 39$, we computed that the $\pfr$-adic $L$-function is a unit for the $\pfr \mid p$ with $\rho_\pfr$ irreducible, so we can conclude that $\Sel_\pfr(A/\Q) = 0$ and hence $\#\Sha(A/\Q)[\pfr] = 0$ from the known cases of the $\GL_2$ IMC. Note that our computation shows that in these cases, the image of $\rho_{\pfr^\infty}$ is maximal, so it contains $\SL_2(\Z_p)$. This implies that the IMC holds \emph{integrally}.
	
	Details will be presented in a forthcoming article, where plan also to extend
	our computations to cover some two-dimensional absolutely simple isogeny
	factors of~$J_0(N)$ that are not Jacobians of quotients of~$X_0(N)$ by
	Atkin-Lehner involutions.
	
	\bibliographystyle{amsplain}
	\bibliography{samplebib}

\begin{thebibliography}{10}

\bibitem{Magma}
Wieb Bosma, John Cannon, and Catherine Playoust.
\newblock {The Magma algebra system. I. The user language}.
\newblock {\em J. Symbolic Comput.}, 24(3–4):235–265, 1997.
\newblock Computational algebra and number theory (London, 1993).

\bibitem{BCDT2001}
Christophe Breuil, Brian Conrad, Fred Diamond, and Richard Taylor.
\newblock {On the modularity of elliptic curves over {$\bold Q$}: wild 3-adic
  exercises}.
\newblock {\em J. Amer. Math. Soc.}, 14(4):843–939, 2001.

\bibitem{BFH1989}
Daniel Bump, Solomon Friedberg, and Jeffrey Hoffstein.
\newblock {A nonvanishing theorem for derivatives of automorphic $L$-functions
  with applications to elliptic curves}.
\newblock {\em Bull. Amer. Math. Soc. (N.S.)}, 21(1):89–93, 1989.

\bibitem{CCSS}
Francesc Castella, Mirela \c{C}iperiani, Christopher Skinner, and Florian
  Sprung.
\newblock {On the Iwasawa main conjectures for modular forms at non-ordinary
  primes}, 2018.
\newblock Preprint, arXiv:1804.10993.

\bibitem{Cojocaru2005}
Alina~Carmen Cojocaru.
\newblock {On the surjectivity of the Galois representations associated to
  non-CM elliptic curves}.
\newblock {\em Canad. Math. Bull.}, 48(1):16–31, 2005.
\newblock With an appendix by Ernst Kani.

\bibitem{CreutzMiller2012}
Brendan Creutz and Robert~L. Miller.
\newblock {Second isogeny descents and the {B}irch and {S}winnerton-{D}yer
  conjectural formula}.
\newblock {\em J. Algebra}, 372:673–701, 2012.

\bibitem{Dieulefait2002}
Luis~V. Dieulefait.
\newblock {Explicit determination of the images of the Galois representations
  attached to abelian surfaces with ${\rm End}(A)=\mathbb Z$}.
\newblock {\em Experiment. Math.}, 11(4):503–512 (2003), 2002.

\bibitem{FLSSSW}
E.~Victor Flynn, Franck Lepr\'{e}vost, Edward~F. Schaefer, William~A. Stein,
  Michael Stoll, and Joseph~L. Wetherell.
\newblock {Empirical evidence for the Birch and Swinnerton-Dyer conjectures for
  modular Jacobians of genus 2 curves}.
\newblock {\em Math. Comp.}, 70(236):1675–1697, 2001.

\bibitem{GJPST}
Grigor Grigorov, Andrei Jorza, Stefan Patrikis, William~A. Stein, and Corina
  Tarni\c{t}\v{a}.
\newblock {Computational verification of the Birch and Swinnerton-Dyer
  conjecture for individual elliptic curves}.
\newblock {\em Math. Comp.}, 78(268):2397–2425, 2009.

\bibitem{GrossZagier1986}
Benedict~H. Gross and Don~B. Zagier.
\newblock {Heegner points and derivatives of $L$-series}.
\newblock {\em Invent. Math.}, 84(2):225–320, 1986.

\bibitem{Hasegawa1995}
Yuji Hasegawa.
\newblock {Table of quotient curves of modular curves $X_0(N)$ with genus $2$}.
\newblock {\em Proc. Japan Acad. Ser. A Math. Sci.}, 71(10):235–239 (1996),
  1995.

\bibitem{KolyvaginLogachev}
V.~A. Kolyvagin and D.~Yu. Logach\"{e}v.
\newblock {Finiteness of the Shafarevich-Tate group and the group of rational
  points for some modular abelian varieties}.
\newblock {\em Algebra i Analiz}, 1(5):171–196, 1989.

\bibitem{Mazur1978}
B.~Mazur.
\newblock {Modular curves and the Eisenstein ideal}.
\newblock {\em Inst. Hautes \'{E}tudes Sci. Publ. Math.}, (47):33–186 (1978),
  1977.
\newblock With an appendix by Mazur and M. Rapoport.

\bibitem{MillerStoll2013}
Robert~L. Miller and Michael Stoll.
\newblock {Explicit isogeny descent on elliptic curves}.
\newblock {\em Math. Comp.}, 82(281):513–529, 2013.

\bibitem{MuellerStoll2016b}
Jan~Steffen M\"uller and Michael Stoll.
\newblock {Canonical heights on genus-2 Jacobians}.
\newblock {\em Algebra Number Theory}, 10(10):2153–2234, 2016.

\bibitem{SchaeferStoll2004}
Edward~F. Schaefer and Michael Stoll.
\newblock {How to do a $p$-descent on an elliptic curve}.
\newblock {\em Trans. Amer. Math. Soc.}, 356(3):1209–1231, 2004.

\bibitem{Skinner2016}
Christopher Skinner.
\newblock {Multiplicative reduction and the cyclotomic main conjecture for
  ${\rm GL}_2$}.
\newblock {\em Pacific J. Math.}, 283(1):171–200, 2016.

\bibitem{SkinnerUrban2014}
Christopher Skinner and Eric Urban.
\newblock {The Iwasawa main conjectures for $\rm GL_2$}.
\newblock {\em Invent. Math.}, 195(1):1–277, 2014.

\bibitem{Stoll2001}
Michael Stoll.
\newblock {Implementing 2-descent for Jacobians of hyperelliptic curves}.
\newblock {\em Acta Arith.}, 98(3):245–277, 2001.

\bibitem{TaylorWiles1995}
Richard Taylor and Andrew Wiles.
\newblock {Ring-theoretic properties of certain {H}ecke algebras}.
\newblock {\em Ann. of Math. (2)}, 141(3):553–572, 1995.

\bibitem{LMFDB}
{The LMFDB collaboration}.
\newblock {{L}-functions and Modular Forms Database}.
\newblock \url{https://www.lmfdb.org/Genus2Curve/Q/}.

\bibitem{vanBommel2019}
Raymond van Bommel.
\newblock {Numerical verification of the Birch and Swinnerton-Dyer conjecture
  for hyperelliptic curves of higher genus over $\Q$ up to squares}.
\newblock {\em Exp. Math.}, page 1–8, 2019.

\bibitem{Waldspurger1985}
J.-L. Waldspurger.
\newblock {Sur les valeurs de certaines fonctions $L$ automorphes en leur
  centre de sym\'{e}trie}.
\newblock {\em Compositio Math.}, 54(2):173–242, 1985.

\bibitem{Wang1995}
Xiang~Dong Wang.
\newblock {$2$-dimensional simple factors of $J_0(N)$}.
\newblock {\em Manuscripta Math.}, 87(2):179–197, 1995.

\bibitem{Wiles1995}
Andrew Wiles.
\newblock {Modular elliptic curves and {F}ermat's last theorem}.
\newblock {\em Ann. of Math. (2)}, 141(3):443–551, 1995.

\end{thebibliography}
\end{document}